\documentclass[review]{elsarticle}
\usepackage{lineno,hyperref}
\modulolinenumbers[5]
\usepackage{epsfig}
\usepackage{geometry}
\usepackage{color}
\usepackage{amsmath}
\usepackage{amssymb}
\usepackage{amsthm}
\usepackage{amsfonts}
\usepackage{enumitem}
\usepackage{dsfont}

\newtheorem{defn}{Definition}[section]
\newtheorem{lem}{Lemma}[section]
\newtheorem{thm}{Theorem}[section]
\newtheorem{assum}{Assumption}[section]
\newtheorem{remark}{Remark}[section]
\newtheorem{corr}{Corollary}[section]
\newtheorem{exam}{Example}[section]

\journal{arXiv}

\begin{document}

\begin{frontmatter}

\title{Stability of stochastic differential equations with respect to time-changed Brownian motions}
%\tnotetext[mytitlenote]{Fully documented templates are available in the elsarticle package on %\href{http://www.ctan.org/tex-archive/macros/latex/contrib/elsarticle}{CTAN}.}

%\author{Qiong Wu}
%\address{Tufts University}

%\author[mymainaddress,mysecondaryaddress]{Qiong Wu}

\author{Qiong Wu\corref{mycorrespondingauthor}}
\cortext[mycorrespondingauthor]{Corresponding author}
\ead{Qiong.Wu@tufts.edu}

\address{Tufts University, \\Department of Mathematics, 503 Boston Avenue, Medford, MA 02155, USA.}

\begin{abstract}
In this paper, the stability behaviors of stochastic differential equations (SDEs) driven by time-changed Brownian motions are discussed. Based on the generalized Lyapunov method and stochastic analysis, necessary conditions are provided for solutions of time-changed SDEs to be stable in different senses, such as stochastic stability, stochastically asymptotic stability and globally stochastically asymptotic stability. Also, a connection between the stability of the solution to the time-changed SDEs and
that to their corresponding non-time-changed SDEs is revealed by applying the duality theorem. Finally, two examples are provided to illustrate the theoretical results. 
\end{abstract}


\end{frontmatter}

%\linenumbers

\section{Introduction}
Stochastic differential equations (SDEs) play an important role in modeling dynamical systems when taking into account uncertainty noise in areas such as physics, biology, economics and finance, see~\cite{Mircea2002, allen2007modeling, shreve2004stochastic, steele2001stochastic}.  One of the most important qualitative behaviors of SDEs is the stability of their solutions, which has been well studied by~\cite{arnold1974stochastic,friedman2014stochastic,khasminskii2011stochastic,mao1994exponential,mao2007stochastic}. When considering SDEs driven by a semimartingale which is the largest class of driven processes for It\^o integrals, Mao studied several kinds of stable behaviors, specifically, exponential stability for linear and nonlinear SDEs~\cite{mao1990exponential,mao1989exponential}, polynomial stability for perturbed SDEs~\cite{mao1992polynomial}, and eventual uniform asymptotic stability for SDEs with a spatial parameter~\cite{mao1991eventual}. Stochastic delay differential equations (SDDEs), also known as stochastic functional differential equations, are a natural generalization of SDEs. They take account into the effects of past behaviors of the system that are imposed on the system's current status. The corresponding stability of solutions to SDDEs has been studied in~\cite{liao2000exponential,mao2000lasalle,ivanov1999stochastic}. To consider the phenomenon of several stochastic systems switching among each other according to the movement of a Markov chain,  a new type of system called an SDE with Markovian switch has been studied in~\cite{mao2000stochastic,mao2006stochastic}. The associated stability analysis{\tiny } of this type of Markovian switching stochastic system is also discussed in~\cite{mao1999stability,mao2000robustness}. \label{key}

Recently, Kei Kobayashi~\cite{Kei2011} has studied the following new type of SDEs driven by a time-changed semimartingale:
\begin{eqnarray}\label{timechangedSDEwrtsemimartingale}
\mathrm{d}X(t) = \rho(t, E_{t}, X(t-))\mathrm{d}t + f(t, E_{t}, X(t-))\mathrm{d}E_{t} + g(t, E_{t}, X(t-))\mathrm{d}Z_{E_{t}},~X(0) = x_{0},
\end{eqnarray}
where $Z_{t}$ is a semimartingale and $E_{t}$ is the time change which is an inverse of a $\beta$-stable subordinator, known as the hitting time (first passage time) of a $\beta$-stable subordinator. For more information about a $\beta$-subordinator and its inverse, see~\cite{bertoin1999subordinators,meerschaert2013inverse}. Since the time change $E_{t}$ and the time-changed semimartingale $Z_{E_{t}}$ in~\eqref{timechangedSDEwrtsemimartingale} are both semimartingales with respect to an appropriate filtration, the time-changed SDE~\eqref{timechangedSDEwrtsemimartingale} is well-posed in the sense of It\^o calculus.
Also, since the inverse process (time change $E_{t}$) remains constant over jump intervals of the subordinator and the duration of those resting periods is not exponential, the time change $E_{t}$ is not Markovian. This non-Markovian time change $E_{t}$ results in a time fractional Fokker-Planck-Kolmogorov (FPK)  equations associated with solutions to SDEs~\eqref{timechangedSDEwrtsemimartingale}. For example, the time fractional FPK equation being associated with the solution to SDE driven by a time-changed L\'evy processes~\cite{hahn2012sdes} and SDE driven by a time-changed Gaussian processes~\cite{hahn2011time} have been studied, respectively. These papers successfully provide a bridge between time fractional FPK 
equations and their corresponding SDEs. 

The SDEs driven by time-changed semimartingales in~\eqref{timechangedSDEwrtsemimartingale} are different from SDEs with respect to semimartingales in~\cite{mao1989exponential,mao1990exponential} where the stabilities of solutions to such SDEs are discussed. In these papers, the SDE with respect to a semimartingale are as the form:
\begin{eqnarray*}
	\mathrm{d}X(t) = AX(t)\mathrm{d}\mu(t) + G(X(t), t)\mathrm{d}M(t),
\end{eqnarray*}
where the semimartingale $M(t) = (M_{1}(t), M_{2}(t), \cdots, M_{m}(t))^{T}$ has the nice quadratic variation property:
\begin{eqnarray}\label{quadraticvariationprocessinMao'spaper}
	\langle M _{i}, M_{j}\rangle(t) = \int_{0}^{t} K_{ij}(s)\mathrm{d}\mu(s),
\end{eqnarray}
where $\mu(s)$ is a continuous nondecreasing process bounded by linear functions. However, even the semimartingale in~\eqref{timechangedSDEwrtsemimartingale} is the nice standard Brownian motion $B_{t}$, its corresponding time-changed Brownian motion $B_{E_{t}}$ dose not satisfy the property~\eqref{quadraticvariationprocessinMao'spaper}. Although the time-changed Brownian motion $B_{E_{t}}$ is a square integrable martingale~\cite{Marcin2010}, its quadratic variation process $E_{t}$ satisfies~\cite{Kei2014}
\begin{eqnarray*}
	\mathbb{E}(E^{n}_{t}) = \frac{n!}{\Gamma(n\beta + 1)}t^{n\beta},
\end{eqnarray*}
where $\Gamma(\cdot)$ is the Gamma function.

In this paper, we study the stability of solutions to the following SDEs which are driven by the time-changed Brownian motion $B_{E_{t}}$
\begin{eqnarray}\label{genearltimechangedsde}
\mathrm{d}X(t) = \rho(t, E_t, X(t))\mathrm{d}t + f(t, E_{t}, X(t))\mathrm{d}E_{t} + g(t, E_t, X(t))\mathrm{d}B_{E_t}
\end{eqnarray}
with a constant initial value $X(0) = x_0$, where functions $\rho, f$ and $g$ $: {\bf R_{+}}\times {\bf R_{+}}\times {\bf R} \to {\bf R}$ are appropriately specified later. 
By applying the time-changed It\^o formula and generalizing a Lyapunov function method discussed in~\cite{mao2007stochastic,mao1991stability}, the stochastic stability, stochastically asymptotic stability and globally stochastically asymptotic stability are discussed. A deep connection between stability of solutions to time-changed SDEs and that of non-time-changed SDEs is derived by applying the duality theorem. Finally, in order to illustrate the theory in this paper, two examples are provided.

\section{Time-changed stochastic differential equations}\label{Pre}
Throughout this paper, let  $(\Omega, \mathcal{F}, (\mathcal{F}_t)_{t\geq 0}, \mathbb{P})$ be a complete filtered probability space with a filtration $(\mathcal{F}_{t})_{t\geq 0}$ which satisfies the usual conditions, i.e., $(\mathcal{F}_{t})_{t\geq 0}$ is right-continuous and $\mathcal{F}_0$ contains all the $\mathbb{P}$-null sets in $\mathcal{F}$.  

An $\mathcal{F}_{t}$ adapted process $E_{t}$ is called a time change if it is a nondecreasing family of $\mathcal{F}_t$ adapted stopping times with c\`{adl\`{a}g} paths which means the sample paths are right-continuous with left limits. Also, if each stopping time is finite almost surely, the time change is said to be finite. Let $E_{t}$ be a finite $\mathcal{F}_{t}$ adapted time change. Definite a new filtration $\mathcal{G}_{t}$ as
\begin{eqnarray*}
	\mathcal{G}_{t} = \mathcal{F}_{E_{t}},
\end{eqnarray*}
which also satisfies the usual conditions because of the right-continuity of $\mathcal{F}_{t}$ and $E_{t}$. 

A subordinator is a one-dimensional L\'evy process with nondecreasing sample paths. For a L\'evy process $U(t)$ on $\bf R$, it is a subordinator if and only if its L\'evy triple takes the form $(b, 0, \nu)$ where $b$ is a nonnegative constant and $\nu$ is a measure on $\bf R$ satisfying 
\begin{eqnarray*}
\nu(-\infty, 0) = 0~\text{and}~ \int_{0}^{\infty}(1\wedge x)\nu(\mathrm{d}x) < \infty.
\end{eqnarray*}
 In particular, a stable subordinator with index $\beta\in(0, 1)$ is a one-dimensional strictly increasing L\'evy process denoted by $U_{\beta}(t)$ and characterized by Laplace transform
\begin{eqnarray}\label{laplacetransformofstableordinator}
	\mathbb{E}[\exp(-sU_{\beta}(t)] = \exp(-ts^{\beta}),  ~~~s\geq 0.
\end{eqnarray}
For more details on L\'evy processes and stable subordinators, please see\cite{Jean1998, applebaum2009levy, Janicki1994}. For an $\mathcal{F}_{t}$ adapted $\beta$-stable subordinator $U_{\beta}(t)$ with Laplace transform as in~\eqref{laplacetransformofstableordinator}, define its generalized inverse $\beta$-stable subordinator as
\begin{eqnarray}\label{definitionofinversesubordinator}
E_{t} := E_{t}^{\beta} = \inf\{s > 0: U_{\beta}(s) > t\},
\end{eqnarray}
which is also called the first hitting time process. Since the $\beta$-stable subordinator $U_{\beta}(t)$ is strictly increasing, its inverse $\beta$-stable subordinator $E_t$ is continuous and nondecreasing. Particularly, the flat periods of the trajectories of $E_{t}$ correspond to the jumps of the $\beta$-stable subordinator $U_{\beta}(t)$. According to~\eqref{definitionofinversesubordinator},  $E_t$ is a stopping time adapted to the filtration $\mathcal{F}_t$, which means $E_{t}$ is an $\mathcal{F}_{t}$ adapted time change.

Let $B_t$ be a standard Brownian motion independent of $E_t$. Define the special filtration $\mathcal{F}_{t}$ as
\begin{eqnarray}\label{filtrationgeneratedbyBandE}
	\mathcal{F}_t = \bigcap\limits_{s > t}\bigg\{\sigma [B_r: 0\leq r \leq s]\vee \sigma [E_r: r \geq 0]\bigg\},
\end{eqnarray}
where $B_{r}$ is the standard Brownian motion and the notation $\sigma_1 \vee \sigma_2$ denotes the $\sigma$-algebra generated by the union of the $\sigma$-algebras $\sigma_1$ and $\sigma_2$. The \cite{Marcin2010} concludes that the time-changed Brownian motion $B_{E_{t}}$ is a square integrable martingale with respect to the filtration $\{\mathcal{G}_t\}_{t\geq 0}$  where $\mathcal{G}_{t} = \mathcal{F}_{E_t}$. Also in~\cite{metzler1999anomalous, metzler2000random}, the transition density function $p(x, t)$ of the time-changed Brownian motion is shown to  satisfy the following fractional diffusion equation
\begin{eqnarray*}
	\frac{\partial p(x, t)}{\partial t} = D_{t}^{1-\beta}\bigg\{\frac{1}{2}\frac{\partial^{2}}{\partial x^{2}}\bigg\}p(x, t),~~p(x, 0) = \delta(x),
\end{eqnarray*}
where $D_{t}^{1-\beta}$ is the Riemann-Liouville fractional derivative defined in term of gamma function $\Gamma(\cdot)$
\begin{eqnarray*}
	D_{t}^{1-\beta}f(t) = \frac{1}{\Gamma(\beta)}\frac{\mathrm{d}}{\mathrm{d}t}\int_{0}^{t}(t - s)^{\beta - 1}f(s)\mathrm{d}s.
\end{eqnarray*}

Now consider the stochastic time-changed differential equation of the form in~\eqref{genearltimechangedsde} where functions $\rho, f$ and $g$ $: {\bf R_{+}}\times {\bf R_{+}}\times {\bf R} \to {\bf R}$ satisfy the following Lipschitz condition as well as a technical assumption.
\begin{assum}\label{assump1}(\cite{Kei2011})
{\bf Lipschitz condition}: there exists a positive constant $K$ such that
\begin{eqnarray*}
	|\rho(t_1, t_2, x) - \rho(t_1, t_2, y)| + |f(t_1, t_2, x) - f(t_1, t_2, y) | + |g(t_1, t_2, x) - g(t_1, t_2, y) | \leq K|x - y|
\end{eqnarray*}
for all $t_1, t_2\in {\bf R_{+}}$ and $x, y\in {\bf R}$. 
\end{assum}

\begin{assum}\label{assump2}(\cite{Kei2011})
	{\bf Technical condition}: if $X(t)$ is a  c\`{adl\`{a}g} and $\mathcal{G}_t$-adapted process, then
\begin{eqnarray*}
	\rho(t, E_t, X(t)), f(t, E_t, X(t)), g(t, E_t, X(t)) \in \mathbb{L}(\mathcal{G}_t),
\end{eqnarray*}
	where $\mathcal{G}_{t} = \mathcal{F}_{E_{t}}$ and $\mathbb{L}(\mathcal{G}_t)$ denotes the class of c\`{agl\`{a}d} (i.e., sample paths which are left-continuous with right limits) and $\mathcal{G}_t$-adapted processes.
\end{assum}
One example of such functions is a linear function $\rho(t_{1}, t_{2}, x) = \rho_{1}(t_{1}, t_{2})x$ where $\rho_{1}(t_{1}, t_{2})$ is a bounded continuous function on ${\bf R_{+}}\times{\bf R_{+}}$. Under Assumptions~\ref{assump1} and~\ref{assump2}, Kobayashi establishes the following lemma in~\cite{Kei2011} on the existence and uniqueness of solutions to the time-changed SDEs~\eqref{genearltimechangedsde}. 
\begin{lem}\label{existenceanduniquenesstheorem}(\cite{Kei2011} Existence and Uniqueness of Solutions)
	Let $E_{t}$ be the inverse of the $\beta$-stable subordinator $U_{\beta}(t)$. Suppose $\rho, f, g$ are real-valued functions defined on ${\bf R_{+}}\times{\bf R_{+}}\times{\bf R}$ satisfying Lipschitz condition~\ref{assump1} and technical condition~\ref{assump2}. Then there exists a unique $\mathcal{G}_{t}$ adapted process $X(t)$ satisfying the time-changed SDE~\eqref{genearltimechangedsde}.
\end{lem}
The next lemma from~\cite{Kei2011} reveals a deep connection between the following classical It\^{o} SDE and its corresponding time-changed SDE. Consider
\begin{eqnarray}\label{classicsde}
	\mathrm{d}Y(t) = f(t, Y(t))\mathrm{d}t + g(t, Y(t))\mathrm{d}B_t, ~~with ~ Y(0) = x_0,
\end{eqnarray}
\begin{eqnarray}\label{timesde}
	\mathrm{d}X(t) = f(E_{t}, X(t))\mathrm{d}E_{t} + g(E_t, X(t))\mathrm{d}B_{E_t}, ~~with ~ X(0) = x_0.
\end{eqnarray}

\begin{lem}\label{duality}
	(\cite{Kei2011} Duality of SDEs) Let $E_{t}$ be the inverse of a $\beta$-stable subordinator $U_{\beta}(t)$.
	\begin{enumerate}
		\item If a process $Y(t)$ satisfies SDE~\eqref{classicsde}, then $X(t) := Y(E_{t})$ satisfies the time-changed SDE~\eqref{timesde}.
		\item If a process $X(t)$ satisfies the time-changed SDE~\eqref{timesde}, then $Y(t) := X(U(t))$ satisfies the classical SDE \eqref{classicsde}.
	\end{enumerate}
\end{lem}

\begin{remark}
Throughout the remainder of the paper,  $X(t, x_0)$ always denotes the unique solution of the time-changed SDE~\eqref{genearltimechangedsde} with the initial value $X(0) = x_0$.
\end{remark}

\section{Stability of stochastic differential equations driven by a time-changed Brownian motion}
In this section, the stability of the solution to the time-changed SDE~\eqref{genearltimechangedsde} is studied. Based on Lyapunov method, some sufficient criteria are proposed to derive stochastically stable solutions to the time-changed SDE~\eqref{genearltimechangedsde}. Before proceeding to the stability analysis, the following useful lemmas, notations, and definitions are introduced.

Let $C^{1, 1, 2}({\bf R_{+}}\times {\bf R_{+}}\times {\bf R} ; {\bf R})$ denote the family of all functions $F(t_{1}, t_{2}, x)$ from ${\bf R_{+}}\times {\bf R_{+}}\times {\bf R}\to{\bf R}$ which are once differentiable in $t_{1} and t_{2}$ as well as continuously twice differentiable in $x$. Then, the generalized time-changed It\^o formula is introduced from~\cite{Qiong2016}.
\begin{lem}\label{timechangeditoformula}(\cite{Qiong2016})
	Suppose $U_{\beta}(t)$ is a $\beta$-stable subordinator and $E_{t}$ is its associated inverse stable subordinator. Let $X(t)$ be a $\mathcal{G}_t = \mathcal{F}_{E_t}$ ($\mathcal{F}_{t}$ defined in~\eqref{filtrationgeneratedbyBandE}) adapted process defined by
	\begin{eqnarray*}
		X(t) = x_{0} + \int_{0}^{t}P(s)\mathrm{d}s + \int_{0}^{t}\Phi(s)\mathrm{d}E_{s} + \int_{0}^{t}\Psi(s)\mathrm{d}B_{E_s},
	\end{eqnarray*}
	where $P, \Phi$ and $\Psi$ are measurable functions such that those integrals exist. If $F: {\bf R_{+}}\times {\bf R_{+}}\times {\bf R} \to {\bf R}$ is a $C^{1, 1, 2}({\bf R_{+}}\times {\bf R_{+}}\times {\bf R} ; {\bf R})$ function, then with probability one
	\begin{eqnarray*}
		\begin{aligned}
			F(t, E_{t}, X(t)) - F(0, 0, x_{0}) &= \int_{0}^{t}F_{t_{1}}(t, E_{s}, X(s))\mathrm{d}s + \int_{0}^{t}F_{t_{2}}(s, E_{s}, X(s))\mathrm{d}E_{s}\\ 
			&+ \int_{0}^{t}F_{x}(s, E_{s}, X(s))P(s)\mathrm{d}s + \int_{0}^{t}F_{x}(s, E_{s}, X(s))\Phi(s)\mathrm{d}E_{s}\\
			& +\int_{0}^{t}F_{x}(s, E_{s}, X(s))\Psi(s)\mathrm{d}B_{E_{s}} + \frac{1}{2}\int_{0}^{t}F_{xx}(s, E_s, X(s))\Psi^{2}(s)\mathrm{d}E_s,
		\end{aligned}
	\end{eqnarray*} 
	where $F_{t_{1}}$, $F_{t_{2}}$ and $F_{x}$ are first derivatives, respectively, and $F_{xx}$ denotes the second derivative.
\end{lem}
Let $C^{1, 1, 2}({\bf R_{+}}\times {\bf R_{+}}\times {\bf R}; {\bf R_{+}})$ denote the family of all nonnegative functions which is a subspace of $C^{1, 1, 2}({\bf R_{+}}\times {\bf R_{+}}\times {\bf R}; {\bf R})$. If $V(t_{1}, t_{2}, x)\in C^{1, 1, 2}({\bf R_{+}}\times {\bf R_{+}}\times {\bf R}; {\bf R_{+}})$, define two operators $L_1V$ and $L_2V$ from ${\bf R_{+}}\times {\bf R_{+}}\times {\bf R}\to {\bf R_{+}}$ by 
\begin{eqnarray*}
	L_1V(t_1, t_2, x) = V_{t_{1}}(t_1, t_2, x) + V_x(t_1, t_2, x)\rho(t_1, t_2, x)
\end{eqnarray*}
and 
\begin{eqnarray*}
	L_2V(t_1, t_2, x) = V_{t_2}(t_1, t_2, x) + V_x(t_1, t_2, x)f(t_1, t_2, x) + \frac{1}{2}V_{xx}g^{2}(t_1, t_2, x),
\end{eqnarray*}
where
\begin{eqnarray*}
	\begin{aligned}
		V_{t_{1}}(t_1, t_2, x) &= \frac{\partial V(t_1, t_2, x)}{\partial t_1},~~~~~~~	V_{t_{2}}(t_1, t_2, x) = \frac{\partial V(t_1, t_2, x)}{\partial t_2},\\
		V_x(t_1, t_2, x) &= \frac{\partial V(t_1, t_2, x)}{\partial x},~~~~~~ V_{xx}(t_1, t_2, x) = \frac{\partial^{2}V(t_1, t_2, x)}{\partial x^{2}}.
	\end{aligned}
\end{eqnarray*}

\begin{defn}(\cite{mao2007stochastic}) 
Assume that 
	\begin{eqnarray*}
		\rho(t_1, t_2, 0) = f(t_1, t_2,  0) = g(t_1, t_2, 0) = 0 ~~~for ~~all ~~t_{1}, t_{2} \geq 0.
	\end{eqnarray*} 
Then the time-changed SDE~\eqref{genearltimechangedsde}  has the solution $X(t, x_{0}) = 0$ corresponding to the initial value $x_0 = 0$. This solution is called the trivial solution or equilibrium position.
\end{defn}

\begin{defn}\label{definitionsofstability}
	(\cite{mao2007stochastic,mao1991stability} Definition 2.1)The trivial solution of the time-changed SDE~\eqref{genearltimechangedsde} is said to be 
	\begin{itemize}
		\item[1)] stochastically stable or stable in probability if for every pair of $\epsilon\in (0, 1)$ and $r >0$, there exists a $\delta = \delta(\epsilon, r) > 0 $ such that 
			\begin{eqnarray*}
				\mathbb{P}\bigg\{|X(t, x_0)| < r~~\textrm{for all}~~t\geq 0\bigg\}\geq 1 - \epsilon
			\end{eqnarray*}
			whenever $|x_0| < \delta$;
		\item[2)] stochastically asymptotically stable if it is stochastically stable and, moreover, for every $\epsilon\in (0, 1)$, there exists a $\delta_{0} = \delta_{0}(\epsilon) > 0$ such that 
		\begin{eqnarray*}
			\mathbb{P}\bigg\{\lim_{t\to \infty} X(t, x_0) = 0 \bigg\}\geq 1 - \epsilon
		\end{eqnarray*}
		whenever $|x_0| < \delta_{0}$;
	   \item[3)] globally stochastically asymptotically stable if it is stochastically stable and, moreover, for all $x_0\in R$  
	   \begin{eqnarray*}
	   	\mathbb{P}\bigg\{\lim_{t\to \infty} X(t, x_0) = 0 \bigg\} = 1.
	   \end{eqnarray*}
	\end{itemize}
\end{defn}

\begin{remark}
Throughout this paper, without loss of generality, consider  the initial value $x_0\in \bf R$ to be a deterministic constant. The initial value being a random variable seems more general but in effect it is equivalent to having a deterministic constant initial value. For more details on this see~ \cite{mao2007stochastic}.
\end{remark}

Let $\mathcal{K}$ denote the family of all nondecreasing functions $\mu: {\bf R_+}\to{\bf R_+}$ such that $\mu(r) > 0$ for all $r > 0$. Also let 
\begin{eqnarray*}
	{\bf S}_h = \{x\in {\bf R}: |x| < h\} ~~ \forall~0< h \leq \infty.
\end{eqnarray*}

\begin{thm}\label{thm1}
	If there exists a function $V(t_1, t_2, x)\in C^{1, 1, 2}({\bf R_{+}}\times {\bf R_{+}}\times {\bf S}_h ; {\bf R_+})$ and $\mu \in \mathcal{K}$ such that the following conditions are satisfied:
	\begin{enumerate}
		\item[1)] $V(t_1, t_2, 0) = 0$,
		\item[2)] $\mu(|x|) \leq V(t_1, t_2, x)$ for all $(t_1, t_2, x) \in {\bf R_+} \times {\bf R_+} \times {\bf S}_h$,
		\item[3)] $L_1V(t_1, t_2, x) \leq 0$,
		\item[4)] $L_2V(t_1, t_2, x) \leq 0$,
	\end{enumerate}
	then the trivial solution of the time-changed SDE~\eqref{genearltimechangedsde} is stochastically stable or stable in probability.
\end{thm}

\begin{proof}
	Let $\varepsilon \in (0, 1) $ and $0 < r < h$ be arbitrary. Since $V(t_1, t_2, x)$ is continuous on $\bf R_{+}\times\bf R_{+}\times\bf S_{h}$ and $V(t_1, t_2, 0) = 0$, there exists a $\delta = \delta(\varepsilon, r) > 0$ such that 
	\begin{eqnarray}\label{supconditioninstochasticstability}
		\frac{1}{\varepsilon}\sup_{x_0\in {\bf S}_\delta} V(0, 0, x_0) \leq \mu(r).
	\end{eqnarray}
From~\eqref{supconditioninstochasticstability} and the condition~$2)$
\begin{eqnarray*}
	\delta < r.
\end{eqnarray*} 
Let $x_0\in {\bf S}_h$ and define the stopping time
\begin{eqnarray}\label{stoppingtimeonr}
	T_r = \inf \{t\geq 0: |X(t, x_0)| \geq r\}
\end{eqnarray}
and
\begin{eqnarray*}
	U_k = k\wedge \inf\bigg\{t\geq 0: \bigg|\int_{0}^{T_r\wedge t}V_x(s, E_s, X(s))g(s, E_s, X(s))\mathrm{d}B_{E_s}\bigg|\geq k \bigg\}
\end{eqnarray*}
for $k = 1, 2, \cdots$. It is obvious that $U_k\to\infty$ as $k\to\infty$. Applying the time-changed It\^{o} formula, Lemma~\ref{timechangeditoformula}, to SDE~\eqref{genearltimechangedsde} yields
\begin{eqnarray}\label{usingItoformulaofV}
	\begin{aligned}
	V(T_r\wedge U_k, E_{T_r\wedge U_k}, &X(T_r\wedge U_k)) = V(0, 0, x_0) + \int_{0}^{T_r\wedge U_k}L_1V(s, E_s, X(s))\mathrm{d}s\\
	 &+ \int_{0}^{T_r\wedge U_k}L_2V(s, E_s, X(s))\mathrm{d}E_s
	+ \int_{0}^{T_r\wedge U_k}V_x(s, E_s, X(s))g(s, E_s, X(s))\mathrm{d}B_{E_s}.
	\end{aligned}
\end{eqnarray}
From \cite{Marcin2010} and \cite{kuo05}, 
\begin{eqnarray*}
	\int_{0}^{t}V_x(s, E_s, X(s))g(s, E_{s}, X(s))\mathrm{d}B_{E_{s}}
\end{eqnarray*}
is a mean 0, square integrable martingale with respect the filtration $\mathcal{G}_{t} = \mathcal{F}_{E_{t}}$. Therefore, taking expectations on both sides of~\eqref{usingItoformulaofV}, and considering conditions $3)$ and $4)$ yield
\begin{eqnarray}\label{expectationofVonItoformulainstochasticstability}
	\mathbb{E}V(T_r\wedge U_k, E_{T_r\wedge U_k}, X(T_r\wedge U_k)) \leq V(0, 0, x_0).
\end{eqnarray}
Let $k\to \infty$ and then apply Fatou's lemma on the left hand side of~\eqref{expectationofVonItoformulainstochasticstability} to obtain 
\begin{eqnarray*}
	\mathbb{E}V(T_r, E_{T_r}, X(T_r)) \leq V(0, 0, x_0).
\end{eqnarray*}
Since $V(t_{1}, t_{2}, x)$ is a nonnegative function, 
\begin{eqnarray}\label{inequalitywithfinitestoppingtimeinstochasticstability}
	\mathbb{E}\bigg\{V(T_r, E_{T_r}, X(T_r))\mathbf{1}_{\{T_r<\infty\}}\bigg\} \leq \mathbb{E}V(T_r, E_{T_r}, X(T_r)) \leq V(0, 0, x_0).
\end{eqnarray}
On the other hand, apply condition $2)$ to yield
\begin{eqnarray}\label{inequalitywithcondition2instochasticstability}
	\mathbb{E}\bigg\{V(T_r, E_{T_r}, X(T_r))\mathbf{1}_{\{T_r<\infty\}}\bigg\} \geq \mathbb{E}\bigg\{\mu(|X(T_{r})| )\mathbf{1}_{\{T_{r} < \infty\}}\bigg\}.
\end{eqnarray}
Also since the function $\mu$ is nondecreasing, applying~\eqref{stoppingtimeonr} and ~\eqref{inequalitywithcondition2instochasticstability} yields 
\begin{eqnarray*}
	|X(T_{r})| \geq r
\end{eqnarray*} and
\begin{eqnarray}\label{inequalitywithcondition2andnondecreasingpropertyinstochasticstability}
\mathbb{E}\bigg\{V(T_r, E_{T_r}, X(T_r))\mathbf{1}_{\{T_r<\infty\}}\bigg\}\geq \mu(r)\mathbb{P}\bigg\{T_r < \infty\bigg\}.
\end{eqnarray}
Then, combining~\eqref{supconditioninstochasticstability},~\eqref{inequalitywithfinitestoppingtimeinstochasticstability} and~\eqref{inequalitywithcondition2andnondecreasingpropertyinstochasticstability} yields
\begin{eqnarray*}
	\mathbb{P}\bigg\{T_r <\infty\bigg\} \leq \varepsilon,
\end{eqnarray*}
which implies
\begin{eqnarray*}
	\mathbb{P}\bigg\{T_r = \infty\bigg\} \geq 1 - \varepsilon.
\end{eqnarray*}
Therefore,
\begin{eqnarray*}
	\mathbb{P}\bigg\{|X(t, x_0)| < r~~\textrm{for all}~~t \geq 0\bigg\} \geq 1 - \varepsilon,
\end{eqnarray*}
which verifies the definition of $X(t, x_{0})$ being stochastically stable and therefore completes the proof of Theorem~\ref{thm1}.
\end{proof}
The next theorem concerns stochastically asymptotically stable as stated in $2)$ of Definition~\ref{definitionsofstability}. Before proposing the next theorem, another helpful notation is needed.
\begin{eqnarray*}
	\underline{{\bf S}_{h}} =  \{x\in {\bf R}: |x| \leq h\} ~~ \forall~0< h \leq \infty.
\end{eqnarray*}

\begin{thm}\label{thm2}
	If there exists a function $V(t_1, t_2, x)\in C^{1, 1, 2}({\bf R_{+}}\times {\bf R_{+}}\times {\bf S}_h; {\bf R_+})$ and $\mu \in \mathcal{K}$ such that 
	\begin{enumerate}
		\item $V(t_1, t_2, 0) = 0$,
		\item $\mu(|x|) \leq V(t_1, t_2, x)$ for all $(t_1, t_2, x) \in {\bf R_+} \times {\bf R_+} \times \bf R$,
		\item for any $\alpha \in (0, h), x\in {\bf S}_h - \underline{{\bf S}_{\alpha}}$, 
	\end{enumerate}
	   assume	
       \begin{eqnarray*}
			L_1V(t_1, t_2, x) \leq -\gamma_{1}(\alpha)~~~a.s.
		\end{eqnarray*}
		and
		\begin{eqnarray*}
			L_2V(t_1, t_2, x) \leq -\gamma_{2}(\alpha)~~~a.s., 
		\end{eqnarray*}
	where $\gamma_1(\alpha) \geq 0$ and $\gamma_{2}(\alpha) \geq 0$ but not equal to zero at the same time. Then the trivial solution of the time-changed SDE~\eqref{genearltimechangedsde} is stochastically asymptotically stable.
\end{thm}

\begin{proof}
From Theorem~\ref{thm1}, it is known that the trivial solution of the time-changed SDE~\eqref{genearltimechangedsde} is stochastically stable. This means for any fixed $\varepsilon \in (0, 1)$, there is a $\delta = \delta(\varepsilon) > 0$ such that 
\begin{eqnarray}\label{meaningofstochasticstabilityinasymptoticallystable}
	\mathbb{P}\bigg\{|X(t, x_0)| < h~~\textrm{for all}~~t\geq 0\bigg\} \geq 1 - \frac{\varepsilon}{5} 
\end{eqnarray}
whenever $x_0\in {\bf S}_{\delta}$. Then fix $x_0\in {\bf S}_{\delta}$ and let $0 < \alpha < \beta < |x_0|$ be arbitrary. Accordingly, define the following stopping times
\begin{eqnarray}\label{stoppingtimeofh}
	T_h = \inf \{t\geq 0: |X(t, x_0)|\geq h\}, 
\end{eqnarray}
\begin{eqnarray*}\label{talpha}
    T_{\alpha} = \inf \{t\geq 0: |X(t, x_0)|\leq \alpha\},
\end{eqnarray*}
and
\begin{eqnarray}\label{sup1}
	U_k = k\wedge \inf\bigg\{t\geq 0: \bigg|\int_{0}^{T_h\wedge T_{\alpha}\wedge t}V_x(s, E_s, X(s))g(s, E_s, X(s))\mathrm{d}B_{E_s}\bigg|\geq k \bigg\}
\end{eqnarray}
for all $k = 1, 2, \cdots$. Obviously, $U_k\to\infty$ as $k\to\infty$ a.s.. Now, applying the time-changed It\^{o} formula, Lemma~\ref{timechangeditoformula}, to SDE~\eqref{genearltimechangedsde}yields
\begin{eqnarray*}
	\begin{aligned}
	&\quad\mathbb{E}V(T_h\wedge T_{\alpha}\wedge U_k, E_{T_h\wedge T_{\alpha}\wedge U_k}, X(T_h\wedge T_{\alpha}\wedge U_k))\\ 
	&= V(0, 0, x_0) + \mathbb{E}\int_{0}^{T_h\wedge T_{\alpha}\wedge U_k}L_1V(s, E_s, X(s))\mathrm{d}s
	+ \mathbb{E}\int_{0}^{T_h\wedge T_{\alpha}\wedge U_k}L_2V(s, E_s, X(s))\mathrm{d}E_s.
	\end{aligned}
\end{eqnarray*}
According to condition $3)$,
\begin{eqnarray*}
	\begin{aligned}
0 &\leq \mathbb{E}V(T_h\wedge T_{\alpha}\wedge U_k, E_{T_h\wedge T_{\alpha}\wedge U_k}, X(T_h\wedge T_{\alpha}\wedge U_k)) \\
&\leq V(0, 0, x_0) - \gamma_1(\alpha)\mathbb{E}(T_h\wedge T_{\alpha}\wedge U_k) - \gamma_2(\alpha)\mathbb{E}(E_{T_h\wedge T_{\alpha}\wedge U_k}).
   \end{aligned}
\end{eqnarray*}
Therefore,
\begin{eqnarray}\label{probabilityinductioninequality}
	\gamma_1(\alpha)\mathbb{E}(T_h\wedge T_{\alpha}\wedge U_k) + \gamma_2(\alpha)\mathbb{E}E_{T_h\wedge T_{\alpha}\wedge U_k}\leq V(0, 0, x_0).
\end{eqnarray}
If $k\to\infty$ on the left hand side of~\eqref{probabilityinductioninequality}, then~\eqref{probabilityinductioninequality} becomes
\begin{eqnarray*}\label{probabilityinductioninequalityreduced}
\gamma_1(\alpha)\mathbb{E}(T_h\wedge T_{\alpha}) + \gamma_2(\alpha)\mathbb{E}E_{T_h\wedge T_{\alpha}}\leq V(0, 0, x_0).
\end{eqnarray*}
Also, since $E_t\to\infty$ as $t\to\infty$, and $\gamma_{1}(\alpha)$ and $\gamma_{2}(\alpha)$ are not equal to zero at the same time, 
\begin{eqnarray}\label{probabilityofunionofstoppingtimesThandTalphafinite}
	\mathbb{P}\bigg\{T_h\wedge T_{\alpha} <\infty \bigg\} = 1.
\end{eqnarray}
From~\eqref{meaningofstochasticstabilityinasymptoticallystable} and~\eqref{stoppingtimeofh}, 
\begin{eqnarray}\label{probabilityofstoppingtimeThfinite}
   \mathbb{P}\bigg\{T_h <\infty\bigg\} \leq \frac{\varepsilon}{5}.
\end{eqnarray}
Combining~\eqref{probabilityofunionofstoppingtimesThandTalphafinite} and~\eqref{probabilityofstoppingtimeThfinite} yields
\begin{eqnarray*}
	\begin{aligned}
	1 = \mathbb{P}\bigg\{T_h\wedge T_{\alpha} < \infty \bigg\} &\leq \mathbb{P}\bigg\{T_h < \infty \bigg\} + \mathbb{P}\bigg\{T_{\alpha} < \infty \bigg\} \leq \mathbb{P}\bigg\{T_{\alpha} < \infty \bigg\} + \frac{\varepsilon}{5}.
	\end{aligned}
\end{eqnarray*}
Consequently,
\begin{eqnarray*}
	\mathbb{P}\bigg\{T_{\alpha} < \infty\bigg\} \geq 1 - \frac{\varepsilon}{5},
\end{eqnarray*}
which implies that there exists a positive constant $\theta = \theta(\alpha)$ such that
\begin{eqnarray*}
	\mathbb{P}\bigg\{T_{\alpha} < \theta \bigg\} \geq 1 - \frac{2\varepsilon}{5}.
\end{eqnarray*}
On the other hand, 
\begin{eqnarray}\label{probabilityofstoppingtimesTalphalessthanThminimumtheta}
	\begin{aligned}
	     \mathbb{P}\bigg\{T_{\alpha} < T_h\wedge\theta\bigg\} &\geq \mathbb{P}\bigg\{\{T_{\alpha} < \theta\}\cap\{T_{h}= \infty \}\bigg\}\\ 
	     &\geq \mathbb{P}\bigg\{T_{\alpha} < \theta \bigg\} - \mathbb{P}\bigg\{T_{h} < \infty \bigg\} \geq 1 - \frac{3\varepsilon}{5}.
	\end{aligned}
\end{eqnarray}
Moreover, define stopping times
\begin{eqnarray}\label{definitionofstoppingtimepi}
		\pi = \left\{ \begin{array}{rl}
		&T_{\alpha}, ~~~~\textrm{if}~~~ T_{\alpha} < T_h\wedge \theta \\
		& \infty, ~~~~\textrm{otherwise}
		\end{array} \right. ,
\end{eqnarray}
\begin{eqnarray}\label{definitionofstoppingtimeTbeta}
	T_{\beta} = \inf\{t > \pi: |X(t, x_0)| \geq \beta \},
\end{eqnarray}
and
\begin{eqnarray}\label{definitionofstoppingtimeTi}
	T_{i} = \inf\bigg\{t\geq \pi: \bigg|\int_{\pi}^{T_{\beta}\wedge t}V_{x}(s, E_s, X(s))g(s, E_s, X(s))\mathrm{d}B_{E_s}\bigg| \geq i\bigg\}
\end{eqnarray}
for $i = 1, 2, \cdots.$ Similarly, $T_{i}\to\infty$ as $i\to\infty$. Again applying the time-changed It\^o formula and condition $3)$ yield
\begin{eqnarray*}
	\mathbb{E}V(\pi\wedge T_h\wedge t, E_{\pi\wedge T_h\wedge t}, X(\pi\wedge T_h\wedge t)) \geq \mathbb{E}V(T_{\beta}\wedge T_{i}\wedge T_h\wedge t, E_{T_{\beta}\wedge T_i\wedge T_h\wedge t}, X(T_{\beta}\wedge T_i\wedge T_h\wedge t))
\end{eqnarray*}
for all $i = 1, 2, \cdots$ and $t\geq 0$. Let $i\to\infty$ and then $t\to\infty$,
\begin{eqnarray}\label{auxiliaryindqualityofstoppingtimes}
		\mathbb{E}V(\pi\wedge T_h, E_{\pi\wedge T_h}, X(\pi\wedge T_h)) \geq \mathbb{E}V(T_{\beta}\wedge T_h,  E_{T_{\beta}\wedge T_h}, X(T_{\beta}\wedge T_h)).
\end{eqnarray}
Also, from~\eqref{definitionofstoppingtimepi},~\eqref{definitionofstoppingtimeTbeta} and~\eqref{auxiliaryindqualityofstoppingtimes},
\begin{eqnarray*}
			\mathbb{E}\bigg\{V(\pi\wedge T_h, E_{\pi\wedge T_h}, X(\pi\wedge T_h))\mathbf{1}_{\{\pi < \infty\}}\bigg\}\geq \mathbb{E}\bigg\{V(T_{\beta}\wedge T_h,  E_{T_{\beta}\wedge T_h}, X(T_{\beta}\wedge T_h))\mathbf{1}_{\{T_{\beta < \infty\}}}\bigg\},
\end{eqnarray*}
which indicates from~\eqref{probabilityofstoppingtimesTalphalessthanThminimumtheta} that
\begin{eqnarray}\label{inequalitybetweenVofTalphaandTbeta} 
	\mathbb{E}\bigg\{V(T_{\alpha}, E_{T_{\alpha}}, X(T_{\alpha}))\mathbf{1}_{\{T_{\alpha} < T_h\wedge\theta\}}\bigg\} \geq \mathbb{E}\bigg\{V(T_{\beta}, E_{T_{\beta}}, X(T_{\beta}))\mathbf{1}_{\big\{\{T_{\beta} < \infty\}\cap\{T_{h} = \infty\}\big\}}\bigg\}.
\end{eqnarray}
Furthermore, define
\begin{eqnarray}\label{definitionofsupofV}
	B_{\alpha} = \sup\bigg\{V(t_{1}, t_{2}, x): (t_{1}, t_{2}, x) \in {\bf R_{+}}\times {\bf R_{+}}\times\underline{{\bf S}_{\alpha}}\bigg\}.
\end{eqnarray}
From condition $1)$,
\begin{eqnarray*}
	\lim_{\alpha \to 0}B_{\alpha} = 0,
\end{eqnarray*}
which means there exists a small $\alpha$ such that 
\begin{eqnarray}\label{relationbetweensupVandmu}
	\frac{B_{\alpha}}{\mu(\beta)} < \frac{\varepsilon}{5}.
\end{eqnarray}
Combining~\eqref{inequalitybetweenVofTalphaandTbeta},~\eqref{definitionofsupofV} and~\eqref{relationbetweensupVandmu} yields
\begin{eqnarray}\label{probabilityupperboundofunionofstoppingtimesTbetanadThinfinity}
	\mathbb{P}\bigg\{\{T_{\beta} < \infty \}\cap\{T_{h} = \infty\}\bigg\} \leq \frac{B_{\alpha}}{\mu(\beta)} < \frac{\varepsilon}{5}.
\end{eqnarray}
Moreover,
\begin{eqnarray}\label{probabilitylowerboundofunionofstoppingtimesTbetanadThinfinity}
	\begin{aligned}
	    \mathbb{P}\bigg\{\{T_{\beta} < \infty \}\cap\{T_{h} = \infty\}\bigg\} &\geq \mathbb{P}\bigg\{T_{\beta} < \infty \bigg\} - \mathbb{P}\bigg\{T_{h} < \infty \bigg\}\\ 
	    &\geq \mathbb{P}\bigg\{T_{\beta} < \infty \bigg\} - \frac{\varepsilon}{5}.
	\end{aligned}
\end{eqnarray}
Taking into account~\eqref{probabilityupperboundofunionofstoppingtimesTbetanadThinfinity} and~\eqref{probabilitylowerboundofunionofstoppingtimesTbetanadThinfinity} yields
\begin{eqnarray}\label{probabilityofstoppingtimeTbetafinite}
	\mathbb{P}\bigg\{T_{\beta} < \infty \bigg\} < \frac{2\varepsilon}{5}.
\end{eqnarray}
Furthermore, combine~\eqref{probabilityofstoppingtimesTalphalessthanThminimumtheta} and~\eqref{probabilityofstoppingtimeTbetafinite} to compute 
\begin{eqnarray*}
	\begin{aligned}
	    \mathbb{P}\bigg\{\pi < \infty~\textrm{and}~T_{\beta} = \infty \bigg\} &\geq \mathbb{P}\bigg\{\pi < \infty \bigg\} - \mathbb{P}\bigg\{T_{\beta} < \infty \bigg\}\\
	    &\geq \mathbb{P}\bigg\{T_{\alpha} < T_{h}\wedge\theta\bigg\} - \frac{2\varepsilon}{5}\\
	    &> 1 -\varepsilon,
	\end{aligned}
\end{eqnarray*}
which implies
\begin{eqnarray*}
	\mathbb{P}\bigg\{\omega: \limsup_{t\to\infty}|X(t, x_0) |\leq \beta\bigg\} > 1 - \varepsilon.
\end{eqnarray*}
Finally, since $\beta$ is arbitrary, let $\beta\to 0$ to yield
\begin{eqnarray*}
	\mathbb{P}\bigg\{\omega: \lim_{t\to\infty}|X(t, x_0) | = 0\bigg\} \geq 1 - \varepsilon.
\end{eqnarray*}
This completes the proof.
\end{proof}

When the initial condition $x_{0}\in\bf R$, the conclusion in Theorem \ref{thm2} still holds. Then, it is possible to prove that the trivial solution of~\eqref{genearltimechangedsde} is globally asymptotically stochastically stable as $h\to\infty$. So, the next theorem concerns globally stochastically asymptotically stable when imposing a condition on the stochastic Lyapunov function $V(t_1, t_2, x)$.

\begin{thm}\label{thm3}
	Assume there exists a function $V(t_1, t_2, x)\in C^{1, 1, 2}({\bf R_{+}}\times {\bf R_{+}}\times {\bf R}; {\bf R_+})$ and $\mu \in \mathcal{K}$ such that 
	\begin{enumerate}
		\item $V(t_1, t_2, 0) = 0$,
		\item $\mu(|x|) \leq V(t_1, t_2, x)$ for all $(t_1, t_2, x) \in {\bf R_+} \times {\bf R_+} \times \bf R$,
		\item for all $(t_1, t_2, x) \in {\bf R_+} \times {\bf R_+} \times \bf R$, 
		\begin{eqnarray*}
			\lim_{|x|\to\infty}\inf_{t_1, t_2\geq 0}V(t_1, t_2, x)	= \infty,
		\end{eqnarray*}
		\item for all $(t_1, t_2, x) \in {\bf R_+} \times {\bf R_+} \times \bf R$,
		\begin{eqnarray*}
			L_1V(t_1, t_2, x) \leq -\gamma_{1}(x)~~~a.s.
		\end{eqnarray*}
		and
		\begin{eqnarray*}
			L_2V(t_1, t_2, x) \leq -\gamma_{2}(x)~~~a.s., 
		\end{eqnarray*}
   \end{enumerate}
  where $\gamma_1(x) \geq 0$ and $\gamma_{2}(x) \geq 0$ but not equal to zero at the same time. Then the trivial solution of the time-changed SDE~\eqref{genearltimechangedsde} is globally stochastically asymptotically  stable.
\end{thm}

\begin{proof}
	From Theorem \ref{thm1}, the trivial solution of the time-changed SDE~\eqref{genearltimechangedsde} is stochastically stable. Then, to show that the trivial solution of SDE~\eqref{genearltimechangedsde} is globally stochastically asymptotically stable, by $3)$ of Definition~\ref{definitionsofstability}, requires proving that
	\begin{eqnarray*}
		\mathbb{P}\bigg\{\lim_{t\to\infty}X(t, x_0) = 0\bigg\} = 1
	\end{eqnarray*}
	for all $x_0\in\bf R$.
	
	For any $x_0\in\bf R$ and arbitrary $\varepsilon\in (0, 1)$, from condition $3)$, there exists a sufficiently large $h$ such that 
	\begin{eqnarray}\label{thm3-ineq1}
		\inf_{|x|\geq h, t_1, t_2\geq 0}V(t_1, t_2, x) \geq \frac{5V(0, 0, x_0)}{\varepsilon}.
	\end{eqnarray}
	Define the stopping time
	\begin{eqnarray*}
		T_h = \inf\{t\geq 0: |X(t, x_0)| \geq h\}.
	\end{eqnarray*}
	Applying the time-changed It\^{o} formula and condition~$4)$ yields
	\begin{eqnarray}\label{thm3-ineq2}
		\mathbb{E}V(T_h\wedge t, E_{T_h\wedge t}, X(T_h\wedge t)) \leq V(0, 0, x_0).
	\end{eqnarray}
	On the other hand, from~\eqref{thm3-ineq1} 
	\begin{eqnarray}\label{expectationoflowerboundofstoppingtatThminimumt}
		\mathbb{E}V(T_h\wedge t, E_{T_h\wedge t}, X(T_h\wedge t)) \geq \frac{5V(0, 0, x_0)}{\varepsilon}\mathbb{P}\bigg\{T_h < t \bigg\}.
	\end{eqnarray}
	Combining~\eqref{thm3-ineq2} and~\eqref{expectationoflowerboundofstoppingtatThminimumt} yields
	\begin{eqnarray*}
		\mathbb{P}\bigg\{T_h < t \bigg\} \leq \frac{\varepsilon}{5}.
	\end{eqnarray*}
	Then, let  $t\to\infty$ to get
	\begin{eqnarray*}
		\mathbb{P}\bigg\{T_h < \infty \bigg\} \leq \frac{\varepsilon}{5},
	\end{eqnarray*}
	which implies
	\begin{eqnarray}\label{probabilityofboundedprocess}
		\mathbb{P}\bigg\{|X(t, x_0)|\leq h~~\textrm{for all}~~t\geq 0\bigg\} \geq 1 - \frac{\varepsilon}{5}.
	\end{eqnarray}
	From the proof of Theorem~\ref{thm2}, \eqref{probabilityofboundedprocess} implies,
	\begin{eqnarray*}
		\mathbb{P}\bigg\{\lim_{t\to\infty}|X(t, x_0)| = 0\bigg\} \geq 1 - \varepsilon.
	\end{eqnarray*}
	Since $\varepsilon$ is arbitrary, 
	\begin{eqnarray*}
		\mathbb{P}\bigg\{\lim_{t\to\infty}|X(t, x_0)| = 0\bigg\} = 1
	\end{eqnarray*}
	for all $x_0\in \bf R$. The proof is complete.	
\end{proof}

\begin{corr}\label{corollary1}
	If the trivial solution of the non-time-changed SDE~\eqref{classicsde} is stochastically stable (stochastically asymptotically stable, globally stochastically asymptotically stable), then the trivial solution of the time-changed SDE~\eqref{timesde} is stochastically stable (stochastically asymptotically stable, globally stochastically asymptotically stable), respectively.
\end{corr}

\begin{proof}
	Denote by $Y(t, x_0)$ the trivial solution of the non-time-changed SDE~\eqref{classicsde} . Since it is stochastically stable, for any pair of $\varepsilon\in(0, 1)$ and $h > 0$, there is a $\sigma = \sigma(\varepsilon, h) > 0$ such that 
	\begin{eqnarray*}
		\mathbb{P}\bigg\{|Y(t, x_0)| < h~~\textrm{for all}~~t \geq 0\bigg\} \geq 1 - \varepsilon
	\end{eqnarray*}
	whenever $|x_0| < \sigma$.
	
	By the Duality Lemma \ref{duality}, the process $X(t, x_{0}) := Y(E_t, x_{0})$ is the trivial solution of the time-changed SDE~\eqref{timesde}. Since $E_{t}$ is independence of Brownian motion $B_t$, by the law of total probability,
		\begin{eqnarray*}
			\begin{aligned}
			   \mathbb{P}\bigg\{|X(t, x_0)| < h~~\textrm{for all}~~t \geq 0\bigg\} &= 
			   \mathbb{P}\bigg\{|Y(E_t, x_0)| < h~~\textrm{for all}~~t \geq 0\bigg\}\\ 
			   &= \int_{0}^{\infty}	
			   \mathbb{P}\bigg\{|Y(E_t, x_0)| < h~~\textrm{for all}~~t \geq 0\bigg|E_t = \tau\bigg\} f_{E_t}(\tau)\mathrm{d}\tau\\
			   &=\int_{0}^{\infty}	
			   \mathbb{P}\bigg\{|Y(\tau, x_0)| < h~~\textrm{for all}~~\tau \geq 0\bigg\} f_{E_t}(\tau)\mathrm{d}\tau\\
			   & \geq 1 - \varepsilon,
			\end{aligned}
	 \end{eqnarray*}
	 which implies that the trivial solution to the time-changed SDE~\eqref{timesde} is stochastically stable. Similarly, we can apply the Duality Lemma~\ref{duality} and the law of total probability to show that the trivial solution of SDE~\eqref{timesde} is stochastically asymptotically stable and globally stochastically asymptotically stable.
\end{proof}

\section{Examples}
In this section, two examples are provided to illustrate the results of the theorems. The first example shows stochastic stability, stochastically asymptotic stability and globally stochastically asymptotic stability . The second example is to verify the deep connection between the stability of the time-changed SDE and that of its corresponding non-time-changed SDE.
\begin{exam}\label{example1inexamplesection}
	Consider a one-dimensional SDE driven by the time-changed Brownian motion
	\begin{eqnarray}\label{sdeinexample1inexamplescetion}
        \mathrm{d}X(t) = -\rho_{1}(t, E_{t})X(t)\mathrm{d}t + f_{1}(t, E_{t})X(t)\mathrm{d}E_{t} + g_{1}(t, E_{t})X(t)\mathrm{d}B_{E_{t}}
	\end{eqnarray}
	with initial value $X(0) = x_{0}$, where $\rho_{1}(t_{1}, t_{2}), f_{1}(t_{1}, t_{2})$ and $g_{1}(t_{1}, t_{2})$ are real-valued functions on $\bf R_{+}\times\bf R_{+}$ satisfying the following condition
	\begin{eqnarray*}
		|\rho_{1}(t_{1}, t_{2})| + |f_{1}(t_{1}, t_{2})| + |g_{1}(t_{1}, t_{2})| \leq L~\text{for all}~t_{1}, t_{2}\in \bf R_{+},
	\end{eqnarray*}
	with positive constant $L$. Also, assume $\rho_{1}(t_{1}, t_{2}), f_{1}(t_{1}, t_{2})$ and $g_{1}(t_{1}, t_{2})$ are measurable with respect to the filtration $\mathcal{G}_{t} = \mathcal{F}_{E_{t}}$. Define the Lyapunov function on $\bf R_{+}\times\bf R_{+}\times R$
	\begin{eqnarray*}
		V(t_{1}, t_{2}, x) = |x|^{\alpha}
	\end{eqnarray*}
	for some $\alpha\in(0, 1)$. Then,
	\begin{eqnarray*}
		L_{1}V(t, E_{t}, x) = -\alpha \rho_{1}(t, E_{t})|x|^{\alpha}
	\end{eqnarray*}
	and
	\begin{eqnarray*}
		L_{2}V(t, E_{t}, x) = -\alpha\bigg(\frac{1}{2}(1-\alpha)g^{2}(t, E_{t})- f_{1}(t, E_{t})\bigg)|x|^{\alpha}.
	\end{eqnarray*}
	Therefore, if 
	\begin{eqnarray}\label{conditionsofexample1inexamplesection}
	\rho_{1}(t, E_{t}) \geq 0\qquad\text{and}\qquad\frac{1}{2}(1-\alpha)g^{2}(t, E_{t}) - f_{1}(t, E_{t}) \geq 0
	\end{eqnarray}
	holds a.s. for all $t, E_{t}\in\bf R_{+}$, Theorem~\ref{thm1} implies that the trivial solution of SDE~\eqref{sdeinexample1inexamplescetion} is stochastically stable. If at least one of the inequalities holds strictly in~\eqref{conditionsofexample1inexamplesection}, Theorem~\ref{thm3} indicates that the trivial solution of SDE~\eqref{sdeinexample1inexamplescetion} is globally stochastically asymptotically stable.
\end{exam}
\begin{exam}
	Consider the following non-time-changed SDE
	\begin{eqnarray}\label{example2inexamplesection}
	    \mathrm{d}Y(t) = f(t, X(t))\mathrm{d}t + g(t, X(t))\mathrm{d}B_{t}
	\end{eqnarray}
	with initial value $X(0) = x_{0}\in\bf R$, where $B_{t}$ is the standard Brownian motion, and $f(t, x), g(t, x): \bf R_{+}\times\bf R\to \bf R$ have the expansions
	\begin{eqnarray}\label{exapansionsoffandginexamplesection}
	f(t, x) = a(t)x + o(|x|), \qquad g(t, x) = b(t)x + o(|x|)
	\end{eqnarray}
	in a neighborhood of $x = 0$ uniformly with respect to $t \geq t_{0}$, where $a(t), b(t)$ are all bounded Borel-measurable real-valued functions. Furthermore, assume the following condition is imposed on $a(t), b(t)$ 
	\begin{eqnarray}\label{conditionofintegralinexamplesection}
		-K \leq \int_{0}^{t}\bigg(a(s) - \frac{1}{2}b^{2}(s) + \theta\bigg)\mathrm{d}s \leq K\qquad\text{for all}~t\geq 0,
	\end{eqnarray}
	where $\theta$ and $K$ are positive. Also let
	\begin{eqnarray}\label{conditionofinequalityrelatedtothetaandalpha}
	    0 < \alpha < \frac{\theta}{\sup_{t\geq 0}b^{2}(t)}
	\end{eqnarray}
	and define the Lyapunov function
	\begin{eqnarray*}
	    V(t, x) = |x|^{\alpha}\exp\bigg(-\alpha\int_{0}^{t}\big(a(s) - \frac{1}{2}b^{2}(s) + \theta\big)\mathrm{d}s\bigg).
	\end{eqnarray*}
From Mao~\cite{mao2007stochastic}, under conditions~\eqref{exapansionsoffandginexamplesection} and~\eqref{conditionofintegralinexamplesection}, the trivial solution of non-time-changed SDE~\eqref{example2inexamplesection} is stochastically asymptotically stable. 

However, from condition~\eqref{exapansionsoffandginexamplesection}, the following time-changed type conditions on $f(E_{t}, x)$ and $g(E_{t}, x)$ still hold
	\begin{eqnarray}\label{exapansionsoftimechangedfandginexamplesection}
	f(E_{t}, x) = a(E_{t})x + o(|x|), \qquad g(E_{t}, x) = b(E_{t})x + o(|x|).
	\end{eqnarray}
On the other hand, from the following change of variables formula in~\cite{Jean1979}
\begin{eqnarray*}
     \int_{0}^{t}\phi(E_{s})\mathrm{d}E_{s} = \int_{0}^{E_{t}}\phi(s)\mathrm{d}s,
\end{eqnarray*}
the condition~\eqref{conditionofintegralinexamplesection} becomes the following time-changed type condition on $a(E_{t})$ and $b(E_{t})$
	\begin{eqnarray}\label{conditionoftimechangedintegralinexamplesection}
		-K \leq \int_{0}^{t}\bigg(a(E_{s}) - \frac{1}{2}b^{2}(E_{s}) + \theta\bigg)\mathrm{d}E_{s} = \int_{0}^{E_{t}}\bigg(a(s) - \frac{1}{2}b^{2}(s) + \theta\bigg)\mathrm{d}s\leq K\qquad\text{for all}~t\geq 0.
	\end{eqnarray}
Define the time-changed Lyapunov function
	\begin{eqnarray*}
		V(E_{t}, x) = |x|^{\alpha}\exp\bigg(-\alpha\int_{0}^{E_{t}}\big(a(s) - \frac{1}{2}b^{2}(s) + \theta\big)\mathrm{d}s\bigg).
	\end{eqnarray*}
By condition~\eqref{conditionoftimechangedintegralinexamplesection},
\begin{eqnarray*}
    |x|^{\alpha}\exp(-\alpha K) \leq V(E_{t}, x) \leq |x|^{\alpha}\exp(\alpha K).
\end{eqnarray*}
Apply the time-changed It\^o formula and compute $L_{2}V(E_{t}, x)$ similarly to~\cite{mao2007stochastic} to yield
\begin{eqnarray*}
	\begin{aligned}
		L_{2}V(E_{t}, t) &= \alpha|x|^{\alpha}\exp\bigg(-\alpha\int_{0}^{E_{t}}(a(s) - \frac{1}{2}b^{2}(s) + \theta)\mathrm{d}s\bigg)\times\bigg(\frac{\alpha}{2}b^{2}(E_{t}) - \theta)\bigg) + o(|x|^{\alpha}).\\
	\end{aligned}
\end{eqnarray*}
	Also since $E_{t}\to\infty$ as $t\to\infty$,
	\begin{eqnarray*}
		b^{2}(E_{t})\leq \sup_{t\geq 0}b^{2}(t),
	\end{eqnarray*}
	which together with ~\eqref{conditionofinequalityrelatedtothetaandalpha} and~\eqref{conditionoftimechangedintegralinexamplesection} implies that
	\begin{eqnarray*}
			L_{2}V(E_{t}, x)&\leq -\frac{1}{2}\alpha\theta\exp(-\alpha K)|x|^{\alpha} + o(|x|^{\alpha}).
	\end{eqnarray*}
	This means $L_{2}V (E_{t}, x)$ is negative-definite in a sufficiently small neighbourhood of $x = 0$ for $t\geq 0$. By Theorem~\ref{thm2}, the trivial solution of SDE~\eqref{example2inexamplesection} is stochastically asymptotically
	stable thereby verifying the Corollary~\ref{corollary1}.
\end{exam}
\section{Acknowledgment}
The author wishes to thank Dr. Marjorie Hahn for her advice, encouragement and patience with my research, and the author's peer Lise Chlebak for her useful discussions. 
%\section*{References}
\bibliography{mybibfile}

\end{document}